\documentclass[12pt]{amsart}
 
\pdfoutput=1
\usepackage[margin=1in]{geometry}
\usepackage{amsmath,amsthm,amssymb,amsrefs,bbm,color,esvect,float,graphicx,mathrsfs}
\usepackage{epstopdf}
\AppendGraphicsExtensions{.tif}

\newenvironment{theorem}[2][Theorem]{\begin{trivlist}
\item[\hskip \labelsep {\bfseries #1}\hskip \labelsep {\bfseries #2.}]}{\end{trivlist}}
\newenvironment{lemma}[2][Lemma]{\begin{trivlist}
\item[\hskip \labelsep {\bfseries #1}\hskip \labelsep {\bfseries #2.}]}{\end{trivlist}}

\newenvironment{remark}[2][Remark]{\begin{trivlist}
\item[\hskip \labelsep {\bfseries #1}\hskip \labelsep {\bfseries #2.}]}{\end{trivlist}}

\theoremstyle{definition}

\begin{document}

\title[Weighted Weak-Type Estimate for Multilinear Calder\'on-Zygmund Operators]{A Weighted Endpoint Weak-Type Estimate for Multilinear Calder\'on-Zygmund Operators}
\author{Cody B. Stockdale}
\address{Cody B. Stockdale, Department of Mathematics, Washington University in St. Louis, One Brookings Drive, St. Louis, MO, 63130, USA}
\email{codystockdale@wustl.edu}

\begin{abstract}
Two proofs of a weighted weak-type $\left(1,\ldots,1;\frac{1}{m}\right)$ estimate for multilinear Calder\'on-Zygmund operators are given. The ideas are motivated by different proofs of the classical weak-type $(1,1)$ estimate for Calder\'on-Zygmund operators. One proof uses the Calder\'on-Zygmund decomposition, and the other proof is motivated by ideas of Nazarov, Treil, and Volberg.

\end{abstract}
\maketitle

\section{Introduction}
The following weak-type $(1,1)$ estimate is essential to the theory of singular integrals.
\begin{theorem}{1}
Let $T$ be a Calder\'on-Zygmund operator. If $f \in L^1(\mathbb{R}^n)$, then $$\|Tf\|_{L^{1,\infty}(\mathbb{R}^n)}:=\sup_{t>0}t|\{|Tf|>t\}|\lesssim\|f\|_{L^1(\mathbb{R}^n)}.$$
\end{theorem}
The original proof of Theorem 1 uses the Calder\'on-Zygmund decomposition of $f\in L^1(\mathbb{R}^n)$, see \cites{Grafakos1, Grafakos2, Stein}. The Calder\'on-Zygmund decomposition method has since become standard for proving endpoint weak-type results for related operators. The Calder\'on-Zygmund decomposition requires the underlying measure space to possess the doubling property: a Borel measure $\mu$ has the \emph{doubling property} if $$\mu(B(x,2r))\lesssim\mu(B(x,r))$$ for all $r>0$ and all $x$ in the space.

Extending the theory to more general settings, Nazarov, Treil, and Volberg gave a new proof of the weak-type $(1,1)$ estimate for Calder\'on-Zygmund operators on nonhomogeneous spaces in \cite {NTV1998}. A nonhomogeneous space is a metric measure space where the underlying measure $\mu$ fails to possess the doubling property, but instead satisfies the polynomial growth condition $$\mu(B(x,r))\lesssim r^n$$ for all $r>0$ and all $x$ in the space. Since Lebesgue measure on $\mathbb{R}^n$ satisfies the polynomial growth condition, the proof in \cite{NTV1998} immediately gives a different proof of Theorem 1.

The Nazarov-Treil-Volberg technique has been studied further. The technique was extended to handle measures with the upper doubling growth condition in \cite{HLYY2012}. It was shown in \cite{S2018} that, if one again assumes the doubling condition, crucial steps in the Nazarov-Treil-Volberg proof of Theorem 1 may be bypassed. Also in \cite{S2018}, an adaptation of the Nazarov-Treil-Volberg argument was used to prove a weighted weak-type $(1,1)$ inequality; and in \cite{SW2019}, an adaptation of the argument was used to prove the weak-type $\left(1,\ldots,1;\frac{1}{m}\right)$ estimate for multilinear Calder\'on-Zygmund operators. The linear weighted estimate was previously proved in \cite{OPR2016} using the Calder\'on-Zygmund decomposition, and the multilinear estimate was first proved in \cite{GT2002}, also using the Calder\'on-Zygmund decomposition. 

In this paper, we combine both of the previously mentioned settings by proving a weighted weak-type $\left(1,\ldots,1;\frac{1}{m}\right)$ estimate for multilinear Calder\'on-Zygmund operators. Two proofs are given -- one uses the Calder\'on-Zygmund decomposition and the other the Nazarov-Treil-Volberg method. See \cite{S2018} for a comparison between the Calder\'on-Zygmund decomposition and Nazarov-Treil-Volberg proofs.

We describe the motivating results. For $1\leq p<\infty$, we say that $w$ is an \emph{$A_p$ weight} if $w$ is locally integrable, positive almost everywhere, and satisfies the $A_p$ condition $$[w]_{A_p}:=\sup_{Q}\left(\frac{1}{|Q|}\int_{Q}w(x)dx\right)\left(\frac{1}{|Q|}\int_{Q}w(x)^{1-p'}dx\right)^{p-1}<\infty;$$ when $p=1$, the quantity $\left(\frac{1}{|Q|}\int_{Q}w(x)^{1-p'}dx\right)^{p-1}$ is interpreted as $\left(\inf_Q w\right)^{-1}$. 
\begin{theorem}{2}
Let $T$ be a Calder\'on-Zygmund operator. If $1\leq p<\infty$, $w\in A_p$, and $f\in L^1(w)$, then $$\|T(fw)w^{-1}\|_{L^{1,\infty}(w)}\lesssim [w]_{A_p}\max\{p,\log(e+[w]_{A_p})\}\|f\|_{L^1(w)}.$$
\end{theorem}
Theorem 2 was proved by Ombrosi, P\'erez, and Recchi in \cite{OPR2016} using the Calder\'on-Zygmund decomposition and by the author in \cite{S2018} using the Nazarov-Treil-Volberg argument. See \cites{CUMP2005, CRR2018, LOP2017, OP2016} for related mixed weak-type inequalities.

Much of the Calder\'on-Zygmund theory was extended to the multilinear setting by Grafakos and Torres in \cite{GT2002}. In particular, they proved the following endpoint weak-type estimate. 
\begin{theorem}{3}
Let $T$ be a multilinear Calder\'on-Zygmund operator. If $f_1,\ldots,f_m\in L^1(\mathbb{R}^n)$, then $$\|T(f_1,\ldots,f_m)\|_{L^{\frac{1}{m},\infty}(\mathbb{R}^n)}\lesssim \prod_{i=1}^m\|f_i\|_{L^1(\mathbb{R}^n)}.$$
\end{theorem}
As in the classical theory, their proof uses the Calder\'on-Zygmund decomposition. Other proofs, also using the Calder\'on-Zygmund decomposition, were later given in the bilinear setting by P\'erez and Torres in \cite{PT2014}, and by Maldonado and Naibo in \cite{MN2009}. Another proof was given by the author and Wick in \cite{SW2019} using a variation of the Nazarov-Treil-Volberg method.

Connecting the weighted theory and the multilinear theory, Lerner, Ombrosi, P\'erez, Torres, and Trujillo-Gonz\'ales introduced the classes of multilinear weights in \cite{LOPTTG2009}. We use the following notation for multilinear $A_{\vec{P}}$ weights: $1\leq p_1,\ldots, p_m \leq \infty$, $\frac{1}{m}\leq p \leq \infty$ satisfies $\frac{1}{p}=\frac{1}{p_1}+\cdots+\frac{1}{p_m}$, $\vec{P}=(p_1,\ldots,p_m)$, $\vec{w}=(w_1,\ldots,w_m)$, and $v_{\vec{w}}=\prod_{i=1}^m w_i^{\frac{p}{p_i}}$. We say $\vec{w} \in A_{\vec{P}}$ if $$[\vec{w}]_{A_{\vec{P}}}:=\sup_{Q}\left(\frac{1}{|Q|}\int_Q v_{\vec{w}}\right)^{\frac{1}{p}}\prod_{i=1}^m\left(\frac{1}{|Q|}\int_Qw_i^{1-p_i'}\right)^{\frac{1}{p_i'}}<\infty;$$ when $p_i=1$, the factor $\left(\frac{1}{|Q|}\int_Qw_i^{1-p_i'}\right)^{\frac{1}{p_i'}}$ is understood as $(\inf_Qw_i)^{-1}$. Note that the quantities $[\vec{w}]_{A_{\vec{P}}}^p$ and $[w]_{A_p}$ coincide when $m=1$. 

We give two proofs of the following theorem.
\begin{theorem}{4}
Let $T$ be a multilinear Calder\'on-Zygmund operator. If $\vec{w}\in A_{(1,\ldots,1)}$ and $f\in L^1(w_i)$ for all $i\in\{1,\ldots,m\}$, then \[\left\|T\left(f_1w_1v_{\vec{w}}^{\frac{1-m}{m}},\ldots,f_mw_mv_{\vec{w}}^{\frac{1-m}{m}}\right)v_{\vec{w}}^{-1}\right\|_{L^{\frac{1}{m},\infty}(v_{\vec{w}})} \lesssim [v_{\vec{w}}]_{A_1}^{2m^2+2m-2}\prod_{i=1}^m\|f_i\|_{L^1(w_i)}.\] 
\end{theorem}
The first proof uses the Calder\'on-Zygmund decomposition and is a weighted version of the proof in \cite{PT2014}; the second proof uses the Nazarov-Treil-Volberg method and is a weighted version of the proof in \cite{SW2019}. See \cite{LOP2019} for a related result that is deduced using multilinear extrapolation.

\begin{remark}{1}  The second proof is actually a weighted version of a simplification of the proof in \cite{SW2019}. Referring to the contents of \cite{SW2019}, the current proof shows that the sets $E_{i,j}$ can be constructed as cubes, that the regularity of Lemma 1 is only required for collections of pairwise disjoint cubes, and that Theorem 2 is not necessary for the weak-type estimate.
\end{remark}

Section 2 describes the definitions and preliminary results, including a weighted version of the regularity condition first described for bilinear kernels in \cite{PT2014}. Section 3 contains two proofs of the main result, Theorem 4.

I would like to thank Brett Wick for his contributions to this article.


\section{Preliminaries}
We use the notation $A \lesssim B$ if there exists $C>0$, possibly depending on $n$, $T$, or $m$, such that $A \leq CB$. The Lebesgue measure of $A\subseteq \mathbb{R}^n$ is denoted by $|A|$, while for a weight $w$, $\int_Aw(x)dx$ is denoted by $w(A)$. The cube with center $x\in \mathbb{R}^n$ and side length $r$ is denoted by $Q(x,r)$. If $Q$ is a cube, then $rQ$ denotes the cube with the same center as $Q$ and side length equal to $r$ times the side length of $Q$.

Let $m$ be a positive integer. We say $K:\mathbb{R}^{n(m+1)}\setminus\{x= y_i \text{ for all } i \in \{1,\ldots,m\}\} \rightarrow \mathbb{C}$ is a \emph{$m$-multilinear Calder\'on-Zygmund kernel} if there exists $\delta >0$ such that the following conditions hold:
\begin{enumerate}
\item (\emph{size}) \[|K(x,y_1,\ldots,y_m)| \lesssim \frac{1}{\left(\sum_{i=1}^m|x-y_i|\right)^{nm}}\]
for all $x,y_1,\ldots,y_m \in \mathbb{R}^n$ with $x\neq y_j$ for some $j$,
\item (\emph{smoothness}) \[|K(x,y_1,\ldots,y_m)-K(x',y_1,\ldots,y_m)| \lesssim \frac{|x-x'|^{\delta}}{(\sum_{i=1}^m|x-y_i|)^{nm+\delta}}\] whenever $\displaystyle |x-x'| \leq \frac{1}{2}\max_{1\leq i\leq m}|x-y_i|$, and 
\[|K(x,y_1,\ldots,y_j,\ldots ,y_m)-K(x,y_1,\ldots, y_j', \ldots ,y_m)|\lesssim \frac{|y_j-y_j'|^{\delta}}{(\sum_{i=1}^m|x-y_i|)^{nm+\delta}}\] for each $j\in \{1,\ldots,m\}$ whenever $\displaystyle |y_j-y_j'| \leq \frac{1}{2}\max_{1\leq i\leq m}|x-y_i|$.
\end{enumerate} 
Let $\mathscr{S}(\mathbb{R}^n)$ denote the space of Schwartz functions on $\mathbb{R}^n$ and $\mathscr{S}'(\mathbb{R}^n)$ the space of tempered distributions on $\mathbb{R}^n$. We say that a $m$-multilinear operator $T:\mathscr{S}(\mathbb{R}^n)\times\cdots\times \mathscr{S}(\mathbb{R}^n) \rightarrow \mathscr{S}'(\mathbb{R}^n)$ is a \emph{multilinear Calder\'on-Zygmund operator associated to a kernel $K$} if $K$ is a $m$-multilinear Calder\'on-Zygmund kernel, if $T$ extends to a bounded operator from $L^{q_1}(\mathbb{R}^n)\times\cdots\times L^{q_m}(\mathbb{R}^n)$ to $L^q(\mathbb{R}^n)$ for some $\frac{1}{m} < q, q_1,\ldots,q_m < \infty$ satisfying $\frac{1}{q}=\frac{1}{q_1}+\cdots+\frac{1}{q_m}$, and if \[T(f_1,\ldots,f_m)(x) = \int_{(\mathbb{R}^n)^m} K(x,y_1,\ldots,y_m)f_1(y_1)\cdots f_m(y_m)dy_1\cdots dy_m\] for compactly supported integrable functions $f_i$ and almost every $x \in \mathbb{R}^n \setminus \bigcap_{i=1}^m\text{supp }f_i$. In all instances that follow, $T$ will represent an $m$-multilinear Calder\'on-Zygmund operator.

The following theorem was proved by Grafakos and Torres in \cite{GT2002}.
\begin{theorem}{5}
If $1< p_1,\ldots,p_m <\infty$, $p$ satisfies $\frac{1}{p}=\frac{1}{p_1}+\cdots+\frac{1}{p_m}$, and $f_i \in L^{p_i}(\mathbb{R}^n)$ for $i\in\{1,\ldots,m\}$, then $$\|T(f_1,\ldots,f_m)\|_{L^p(\mathbb{R}^n)}\lesssim\prod_{i=1}^m\|f_i\|_{L^{p_i}(\mathbb{R}^n)}.$$
\end{theorem}
We will use Theorem 5 in the proofs of Theorem 4 when $p_1=\cdots=p_m=m$ and $p=1$.

A characterization of the multilinear $A_{\vec{P}}$ condition in terms of linear $A_q$ conditions was established in \cite{LOPTTG2009}.
\begin{theorem}{6}
The following conditions are equivalent:
\begin{enumerate}
\item $\vec{w}\in A_{\vec{P}}$.
\item $v_{\vec{w}}\in A_{mp}$ and $w_i^{1-p_i'} \in A_{mp_i'}$ for all $i\in\{1,\ldots,m\}$. When $p_i=1$, the condition $w_i^{1-p_i'} \in A_{mp_i'}$ is understood as $w_i^{\frac{1}{m}} \in A_1$.
\end{enumerate}
\end{theorem}

\begin{remark}{2}
Tracking down the estimates in the proof of Theorem 6 gives the relationships $$[v_{\vec{w}}]_{A_{mp}}\leq[\vec{w}]_{A_{\vec{P}}}^p,$$ $$\left[w_i^{1-p_i'}\right]_{A_{mp_i'}}\leq [\vec{w}]_{A_{\vec{P}}}^{p_i'} \,\, \text{when } p_i>1,$$ $$\left[w_i^{\frac{1}{m}}\right]_{A_{1}}\leq [\vec{w}]_{A_{\vec{P}}}^{\frac{1}{m}} \,\, \text{when } p_i=1, \,\, \text{and}$$ $$[\vec{w}]_{A_{\vec{P}}}\leq [v_{\vec{w}}]_{A_{mp}}^{\frac{1}{p}}\prod_{i=1}^m\left[w_i^{1-p_i'}\right]_{A_{mp_i'}}^{\frac{1}{p_i'}},$$ where $\left[w_i^{1-p_i'}\right]_{A_{mp_i'}}^{\frac{1}{p_i'}}$ is interpreted as $\left[w_i^{\frac{1}{m}}\right]_{A_1}^m$ when $p_i=1$.
\end{remark}

We will use the following property of $A_1$ weights. 
\begin{lemma}{1}
If $w \in A_1$ and $0\leq \gamma\leq 1$, then $w^{\gamma}\in A_1$ with $[w^{\gamma}]_{A_1}\leq[w]_{A_1}^{\gamma}$.
\end{lemma}

\begin{proof}
The cases when $\gamma=0$ and $\gamma=1$ are clear. If $0<\gamma<1$, then $\frac{1}{\gamma}>1$. Applying H\"older's inequality and the $A_1$ condition of $w$ gives $$\frac{1}{|Q|}\int_Qw(y)^{\gamma}dy \leq 
\left(\frac{1}{|Q|}\int_Qw(y)dy\right)^{\gamma}
\left(\frac{1}{|Q|}\int_Q1^{\left(\frac{1}{\gamma}\right)'}dy\right)^{\frac{1}{\left(\frac{1}{\gamma}\right)'}}\leq [w]_{A_1}^{\gamma}\left(\inf_Q w^{\gamma}\right).$$
\end{proof}

We will use the following maximal function in the second proof of the main theorem in Section 3. Given a weight $w$, define the \emph{uncentered maximal function associated to $w$} by $$M_wf(x):=\sup_{Q\ni x}\frac{1}{w(Q)}\int_Q|f(y)|w(y)dy.$$
\begin{lemma}{2}
If $w \in A_p$ and $f\in L^1(w)$, then $$\|M_wf\|_{L^{1,\infty}(w)}\lesssim \|f\|_{L^1(w)}.$$
\end{lemma}
The operator norm of $M_w$ does not depend on the $A_p$ characteristic of $w$.

The following lemma is well-known and proved in \cite{Grafakos1,Grafakos2,Stein}. 
\begin{lemma}{3}
Let $k:[0,\infty)\rightarrow [0,\infty)$ be decreasing and continuous except at a finite number of points. If $K(x)=k(|x|)$ is in $L^1(\mathbb{R}^n)$, then for all $f \in L_{\text{loc}}^1(\mathbb{R}^n)$, $$(|f|*K)(x)\leq \|K\|_{L^1(\mathbb{R}^n)}M(f)(x),$$
where $M$ denotes the classical Hardy-Littlewood maximal operator.
\end{lemma}

The following lemma is a weighted version of the multilinear geometric H\"ormander condition first introduced in the bilinear setting in \cite{PT2014} and generalized in the multilinear setting in \cite{SW2019}. We use the following vector notations $\vv{y}_{i,k}=(y_i,y_{i+1},\ldots,y_k)$ and $\vv{c}_{(i,j_i),(k,j_k)}=(c_{i,j_i},c_{i+1,j_{i+1}},\ldots,c_{k,j_k})$.

\begin{lemma}{4}
If $w\in A_{1}$, $l\in\{1,\ldots,m\},$ and each of  $\mathcal{Q}_1,\ldots,\mathcal{Q}_l$ consists of pairwise disjoint cubes  $\mathcal{Q}_i=\{Q_{i,1},Q_{i,2},\ldots\}$ where $Q_{i,j}=Q(c_{i,j},r_{i,j})$, then
\begin{align*}
&\sum_{j_1,\ldots,j_l=1}^{\infty}\prod_{i=1}^lw^{\frac{1}{m}}(Q_{i,j_i})\int_{\mathbb{R}^{n(m-l)}}\prod_{i=l+1}^{m}w(y_i)^{\frac{1}{m}}\\
&\quad\quad\times\sup_{\substack{(y_1,\ldots,y_l) \\\in \prod_{i=1}^lQ_{i,j_i}}} \int_{\mathbb{R}^n\setminus\left(\bigcup_{i=1}^l \Omega_{i}^{*}\right)} |K(x,\vv{y}_{1,m})-K(x,\vv{c}_{(1,j_1),(l,j_l)},\vv{y}_{l+1,m})|dxd\vv{y}_{l+1,m}\\
&\lesssim [w]_{A_{1}}^{\frac{2m-2}{m}}\sum_{i=1}^{l} w\left(\Omega_{i}\right)
\end{align*} 
where $\Omega_i:= \bigcup_{j=1}^{\infty}Q_{i,j}$ and $\Omega_i^{*}:= \bigcup_{j=1}^{\infty}2\sqrt{n}Q_{i,j}$.
\end{lemma} 
It is not important that the indices of the $\mathcal{Q}_i$ range from $1$ to $l$ -- a symmetric proof yields the lemma whenever the set of indices is a nonempty subset of $\{1,\ldots,m\}$.

\begin{proof} For $i=1,\ldots, l$, fix $Q_{i,j_i}\in \mathcal{Q}_i$. Use the smoothness condition of $K$ to see
\begin{align*} 
\sup_{\substack{(y_1,\ldots,y_l) \\\in \prod_{i=1}^lQ_{i,j_i}}} &\int_{\mathbb{R}^n\setminus\left(\bigcup_{i=1}^l \Omega_{i}^{*}\right)}|K(x,\vv{y}_{1,m})-K(x,\vv{c}_{(1,j_1),(l,j_l)},\vv{y}_{l+1,m})|dx\\
&\lesssim \sup_{\substack{(y_1,\ldots,y_l) \\\in \prod_{i=1}^lQ_{i,j_i}}} \int_{\mathbb{R}^n\setminus\left(\bigcup_{i=1}^l \Omega_{i}^{*}\right)} \frac{\sum_{i=1}^l|y_i-c_{i,j_i}|^{\delta}}{(\sum_{i=1}^m|x-y_i|)^{nm+\delta}}dx\\
&\lesssim \sup_{\substack{(y_1,\ldots,y_l) \\\in \prod_{i=1}^lQ_{i,j_i}}} \int_{\mathbb{R}^n\setminus\left(\bigcup_{i=1}^l \Omega_{i}^{*}\right)} \frac{\sum_{i=1}^lr_{i,j_i}^{\delta}}{(\sum_{i=1}^m|x-y_i|)^{nm+\delta}}dx.
\end{align*}
Since for fixed $y_i \in\overline{Q_{i,j_i}}$, $i=l+1,\ldots,m$, the function \\$\displaystyle\int_{\mathbb{R}^n\setminus\left(\bigcup_{i=1}^l \Omega_{i}^{*}\right)} \frac{\sum_{i=1}^lr_{i,j_i}^{\delta}}{(\sum_{i=1}^m|x-y_i|)^{nm+\delta}}dx$ is continuous in the variables $y_i\in \overline{Q_{i,j_i}}$, $i=1,\ldots, l$, we may write 

\begin{align*}
\sup_{\substack{(y_1,\ldots,y_l) \\\in \prod_{i=1}^lQ_{i,j_i}}} &\int_{\mathbb{R}^n\setminus\left(\bigcup_{i=1}^l \Omega_{i}^{*}\right)} \frac{\sum_{i=1}^lr_{i,j_i}^{\delta}}{(\sum_{i=1}^m|x-y_i|)^{nm+\delta}}dx\\
&=\int_{\mathbb{R}^n\setminus\left(\bigcup_{i=1}^l \Omega_{i}^{*}\right)} \frac{\sum_{i=1}^lr_{i,j_i}^{\delta}}{(\sum_{i=1}^l|x-y_i^*|+\sum_{i=l+1}^m|x-y_i|)^{nm+\delta}}dx,
\end{align*}
and
\begin{align*}
\inf_{\substack{(y_1,\ldots,y_l) \\\in \prod_{i=1}^lQ_{i,j_i}}} &\int_{\mathbb{R}^n\setminus\left(\bigcup_{i=1}^l \Omega_{i}^{*}\right)} \frac{\sum_{i=1}^lr_{i,j_i}^{\delta}}{(\sum_{i=1}^m|x-y_i|)^{nm+\delta}}dx\\
&=\int_{\mathbb{R}^n\setminus\left(\bigcup_{i=1}^l \Omega_{i}^{*}\right)} \frac{\sum_{i=1}^lr_{i,j_i}^{\delta}}{(\sum_{i=1}^l|x-y_{i_*}|+\sum_{i=l+1}^m|x-y_i|)^{nm+\delta}}dx.
\end{align*}
Note that for $x\in\mathbb{R}^n\setminus\left(\bigcup_{i=1}^l \Omega_{i}^{*}\right)$, $|x-y_{i_*}|\leq \sqrt{n}r_{i,j_i}+|x-y_i^*|$ and $|x-y_i^*|\ge \frac{1}{2}r_{i,j_i}$, so \[\frac{\sum_{i=1}^l|x-y_{i_*}|+\sum_{i=l+1}^m|x-y_i|}{\sum_{i=1}^l|x-y_i^*|+\sum_{i=l+1}^m|x-y_i|}\leq \frac{\sum_{i=1}^l\sqrt{n}r_{i,j_i}}{\sum_{i=1}^l|x-y_i^*|+\sum_{i=l+1}^m|x-y_i|}+1\leq 2\sqrt{n}+1.\]
Then
\begin{align*}
&\sup_{\substack{(y_1,\ldots,y_l) \\\in \prod_{i=1}^lQ_{i,j_i}}} \int_{\mathbb{R}^n\setminus\left(\bigcup_{i=1}^l \Omega_{i}^{*}\right)}|K(x,\vv{y}_{1,m})-K(x,\vv{c}_{(1,j_1),(l,j_l)},\vv{y}_{l+1,m})|dx\\
&\quad\lesssim \int_{\mathbb{R}^n\setminus\left(\bigcup_{i=1}^l \Omega_{i}^{*}\right)} \frac{\sum_{i=1}^lr_{i,j_i}^{\delta}}{(\sum_{i=1}^l|x-y_{i_*}|+\sum_{i=l+1}^m|x-y_{i}|)^{nm+\delta}}\\
&\quad\quad\times\left(\frac{\sum_{i=1}^l|x-y_{i_*}|+\sum_{i=l+1}^m|x-y_i|}{\sum_{i=1}^l|x-y_i^*|+\sum_{i=l+1}^m|x-y_i|}\right)^{nm+\delta}dx\\
&\quad\lesssim \int_{\mathbb{R}^n\setminus\left(\bigcup_{i=1}^l \Omega_{i}^{*}\right)} \frac{\sum_{i=1}^lr_{i,j_i}^{\delta}}{(\sum_{i=1}^l|x-y_{i_*}|+\sum_{i=l+1}^m|x-y_i|)^{nm+\delta}}dx\\
&\quad=\inf_{\substack{(y_1,\ldots,y_l) \\\in \prod_{i=1}^lQ_{i,j_i}}} \int_{\mathbb{R}^n\setminus\left(\bigcup_{i=1}^l \Omega_{i}^{*}\right)} \frac{\sum_{i=1}^lr_{i,j_i}^{\delta}}{(\sum_{i=1}^m|x-y_{i}|)^{nm+\delta}}dx.
\end{align*}

Using the previous estimate, Fubini's theorem, and trivial estimates, we get the bound
\begin{align*}
&\sum_{j_1,\ldots,j_l=1}^{\infty}\prod_{i=1}^lw^{\frac{1}{m}}(Q_{i,j_i})\int_{\mathbb{R}^{n(m-l)}}\prod_{i=l+1}^{m}w(y_i)^{\frac{1}{m}}\\
&\quad\quad\times\sup_{\substack{(y_1,\ldots,y_l) \\\in \prod_{i=1}^lQ_{i,j_i}}} \int_{\mathbb{R}^n\setminus\left(\bigcup_{i=1}^l \Omega_{i}^{*}\right)} |K(x,\vv{y}_{1,m})-K(x,\vv{c}_{(1,j_1),(l,j_l)},\vv{y}_{l+1,m})|dxd\vv{y}_{l+1,m}\\
&\lesssim \sum_{j_1,\ldots,j_l=1}^{\infty}\prod_{i=1}^lw^{\frac{1}{m}}(Q_{i,j_i})\int_{\mathbb{R}^{n(m-l)}}\prod_{i=l+1}^{m}w(y_i)^{\frac{1}{m}}\\
&\quad\quad\times\inf_{\substack{(y_1,\ldots,y_l) \\\in \prod_{i=1}^lQ_{i,j_i}}} \int_{\mathbb{R}^n\setminus\left(\bigcup_{i=1}^l \Omega_{i}^{*}\right)}\frac{\sum_{i=1}^lr_{i,j_i}^{\delta}}{(\sum_{i=1}^m|x-y_i|)^{nm+\delta}}dxd\vv{y}_{l+1,m}\\
&\leq \sum_{j_1,\ldots,j_l=1}^{\infty}\int_{\mathbb{R}^{n(m-l)}}\int_{Q_{l,j_l}}\cdots\int_{Q_{1,j_1}}\prod_{i=1}^mw(y_i)^{\frac{1}{m}}\int_{\mathbb{R}^n\setminus\left(\bigcup_{i=1}^l \Omega_{i}^{*}\right)}\frac{\sum_{i=1}^lr_{i,j_i}^{\delta}}{(\sum_{i=1}^m|x-y_i|)^{nm+\delta}}dxd\vv{y}_{1,m}\\
&=\sum_{k=1}^l\left(\sum_{\substack{j_1,\ldots,j_l=1 \\ r_{k,j_k}\ge r_{i,j_i}\,\text{all }i}}^{\infty}\int_{\mathbb{R}^{n(m-l)}}\int_{Q_{l,j_l}}\cdots\int_{Q_{1,j_1}}\prod_{i=1}^mw(y_i)^{\frac{1}{m}}\right.\\
&\quad\quad\times\left.\int_{\mathbb{R}^n\setminus\left(\bigcup_{i=1}^l \Omega_{i}^{*}\right)}\frac{\sum_{i=1}^lr_{i,j_i}^{\delta}}{(\sum_{i=1}^m|x-y_i|)^{nm+\delta}}dxd\vv{y}_{1,m}\right).
\end{align*}

We will control the term of the summation above with $k=1$; the other terms are handled similarly. Using trivial estimates, Fubini's theorem, and the fact that the $Q_{i,j_i}$ have disjoint interiors, we obtain
\begin{align*}
&\sum_{\substack{j_1,\ldots,j_l=1 \\ r_{1,j_1}\ge r_{i,j_i}\,\text{all }i}}^{\infty}\int_{\mathbb{R}^{n(m-l)}}\int_{Q_{l,j_l}}\cdots\int_{Q_{1,j_1}}\prod_{i=1}^mw(y_i)^{\frac{1}{m}}\int_{\mathbb{R}^n\setminus\left(\bigcup_{i=1}^l \Omega_{i}^{*}\right)}\frac{\sum_{i=1}^lr_{i,j_i}^{\delta}}{(\sum_{i=1}^m|x-y_i|)^{nm+\delta}}dxd\vv{y}_{1,m}\\
&\quad\lesssim \sum_{j_1,\ldots,j_l=1}^{\infty}\int_{\mathbb{R}^{n(m-l)}}\int_{Q_{l,j_l}}\cdots\int_{Q_{1,j_1}}\prod_{i=1}^mw(y_i)^{\frac{1}{m}}\int_{\mathbb{R}^n\setminus\left(\bigcup_{i=1}^l \Omega_{i}^{*}\right)}\frac{r_{1,j_1}^{\delta}}{(\sum_{i=1}^m|x-y_i|)^{nm+\delta}}dxd\vv{y}_{1,m}\\
&\quad\leq \sum_{j_1=1}^{\infty}\int_{Q_{1,j_1}}w(y_1)^{\frac{1}{m}}\int_{\mathbb{R}^n\setminus\left(\bigcup_{i=1}^l \Omega_{i}^{*}\right)}\frac{r_{1,j_1}^{\delta}}{|x-y_1|^{\delta}}\int_{\mathbb{R}^{n(m-l)}}\\
&\quad\quad\times\sum_{j_2,\ldots,j_l=1}^{\infty}\int_{Q_{l,j_l}}\cdots\int_{Q_{2,j_2}}\frac{\prod_{i=2}^mw(y_i)^{\frac{1}{m}}}{(\sum_{i=1}^m|x-y_i|)^{nm}}d\vv{y}_{2,m}dxdy_1\\
&\quad\leq \sum_{j_1=1}^{\infty}\int_{Q_{1,j_1}}w(y_1)^{\frac{1}{m}}\int_{\mathbb{R}^n\setminus\left(\bigcup_{i=1}^l \Omega_{i}^{*}\right)}\frac{r_{1,j_1}^{\delta}}{|x-y_1|^{\delta}}\int_{\mathbb{R}^{n(m-1)}}\frac{\prod_{i=2}^mw(y_i)^{\frac{1}{m}}}{(\sum_{i=1}^m|x-y_i|)^{nm}}d\vv{y}_{2,m}dxdy_1.
\end{align*}

Repeatedly use Lemma 3 first with $K=\frac{1}{(\sum_{i=1}^{m-1}|x-y_i|+|\cdot|)^{nm}}$, second with \\$K=\frac{1}{(\sum_{i=1}^{m-2}|x-y_i|+|\cdot|)^{n(m-1)}}$, etcetera, and the fact that $\displaystyle M\left(w^{\frac{1}{m}}\right)(x)\leq\left[w^{\frac{1}{m}}\right]_{A_1}w(x)^{\frac{1}{m}}$ (which is true by Lemma 1) to control the above expression by
\begin{align*}
&\sum_{j_1=1}^{\infty}\int_{Q_{1,j_1}}w(y_1)^{\frac{1}{m}}\int_{\mathbb{R}^n\setminus\left(\bigcup_{i=1}^l \Omega_{i}^{*}\right)}\frac{r_{1,j_1}^{\delta}M\left(w^{\frac{1}{m}}\right)(x)}{|x-y_1|^{\delta}}\int_{\mathbb{R}^{n(m-2)}}\frac{\prod_{i=2}^{m-1}w(y_i)^{\frac{1}{m}}}{(\sum_{i=1}^{m-1}|x-y_i|)^{n(m-1)}}d\vv{y}_{2,m-1}dxdy_1\\
&\quad\leq\left[w^{\frac{1}{m}}\right]_{A_1}\sum_{j_1=1}^{\infty}\int_{Q_{1,j_1}}w(y_1)^{\frac{1}{m}}\int_{\mathbb{R}^n\setminus\left(\bigcup_{i=1}^l \Omega_{i}^{*}\right)}\frac{r_{1,j_1}^{\delta}w(x)^{\frac{1}{m}}M\left(w^{\frac{1}{m}}\right)(x)}{|x-y_1|^{\delta}}\\
&\quad\quad\quad\times\int_{\mathbb{R}^{n(m-3)}}\frac{\prod_{i=2}^{m-2}w(y_i)^{\frac{1}{m}}}{(\sum_{i=1}^{m-2}|x-y_i|)^{n(m-2)}}d\vv{y}_{2,m-2}dxdy_1\\
&\quad\leq\cdots\\
&\quad\leq\left[w^{\frac{1}{m}}\right]_{A_1}^{m-1}\sum_{j_1=1}^{\infty}\int_{Q_{1,j_1}}w(y_1)^{\frac{1}{m}}\int_{\mathbb{R}^n\setminus\left(\bigcup_{i=1}^l \Omega_{i}^{*}\right)}\frac{r_{1,j_1}^{\delta}w(x)^{\frac{m-1}{m}}}{|x-y_1|^{n+\delta}}dxdy_1\\
&\quad\leq\left[w^{\frac{1}{m}}\right]_{A_1}^{m-1}\sum_{j_1=1}^{\infty}\int_{Q_{1,j_1}}w(y_1)^{\frac{1}{m}}\int_{\left\{|x-y_1|>\frac{1}{2}r_{1,j_1}\right\}}\frac{r_{1,j_1}^{\delta}w(x)^{\frac{m-1}{m}}}{|x-y_1|^{n+\delta}}dxdy_1.
\end{align*}
Use Lemma 3 with $K=\frac{r_{1,j_1}^{\delta}}{|\cdot|^{n+\delta}}\mathbbm{1}_{\left\{|\cdot|>\frac{1}{2}r_{1,j_1}\right\}}$, the fact that $\|K\|_{L^1(\mathbb{R}^n)}\lesssim 1$, the fact that $M\left(w^{\frac{m-1}{m}}\right)(y_1)\leq \left[w^{\frac{m-1}{m}}\right]_{A_1}w^{\frac{m-1}{m}}(y_1)$ (which is true by Lemma 1), the estimates in Remark 2, and the pairwise disjointness of $Q_{1,j_1}$ to further estimate the previous expression by a constant multiplied by
\begin{align*}
&\left[w^{\frac{1}{m}}\right]_{A_1}^{m-1}\sum_{j_1=1}^{\infty}\int_{Q_{1,j_1}}w(y_1)^{\frac{1}{m}}M\left(w^{\frac{m-1}{m}}\right)(y_1)dy_1\\
&\quad\leq\left[w^{\frac{1}{m}}\right]_{A_1}^{m-1}\left[w^{\frac{m-1}{m}}\right]_{A_1}\sum_{j_1=1}^{\infty}\int_{Q_{1,j_1}}w(y_1)dy_1\\
&\quad\leq [w]_{A_{1}}^{\frac{2m-2}{m}}w(\Omega_1).
\end{align*}
Similarly, for $k=2,\ldots,l$, 
\begin{align*}
&\sum_{\substack{j_1,\ldots,j_l=1 \\ r_{k,j_k}\ge r_{i,j_i}\,\text{all }i}}^{\infty}\int_{\mathbb{R}^{n(m-l)}}\int_{Q_{l,j_l}}\cdots\int_{Q_{1,j_1}}\prod_{i=1}^mw(y_i)^{\frac{1}{m}}\int_{\mathbb{R}^n\setminus\left(\bigcup_{i=1}^l \Omega_{i}^{*}\right)}\frac{\sum_{i=1}^lr_{i,j_i}^{\delta}}{(\sum_{i=1}^m|x-y_i|)^{nm+\delta}}dxd\vv{y}_{1,m}\\
&\quad\lesssim\left[w^{\frac{1}{m}}\right]_{A_1}^{m-1}\left[w^{\frac{m-1}{m}}\right]_{A_1}w\left(\Omega_k\right)\\
&\quad \leq [w]_{A_{1}}^{\frac{2m-2}{m}}w\left(\Omega_k\right).
\end{align*}
This completes the proof.
\end{proof}


\section{Main Results}
We give two proofs of the main result. The first proof uses the Calder\'on-Zygmund decomposition and the second proof uses the Nazarov-Treil-Volberg method. Recall that, for a measure $\mu$, the quasinorm $\|\cdot\|_{L^{p,\infty}(\mu)}$ is given by $\|f\|_{L^{p,\infty}(\mu)}^p:=\sup_{t>0}t^p\mu(\{|f|>t\})$.
\begin{theorem}{4}
If $\vec{w}\in A_{(1,\ldots,1)}$ and $f\in L^1(w_i)$ for all $i\in\{1,\ldots,m\}$, then \[\left\|T\left(f_1w_1v_{\vec{w}}^{\frac{1-m}{m}},\ldots,f_mw_mv_{\vec{w}}^{\frac{1-m}{m}}\right)v_{\vec{w}}^{-1}\right\|_{L^{\frac{1}{m},\infty}(v_{\vec{w}})} \lesssim [v_{\vec{w}}]_{A_1}^{2m^2+2m-2}\prod_{i=1}^m\|f_i\|_{L^1(w_i)}.\] 
\end{theorem}

\begin{proof}[Proof 1]
Let $t>0$ be given. We will show that $$v_{\vec{w}}\left(\left\{\left|T\left(f_1w_1v_{\vec{w}}^{\frac{1-m}{m}},\ldots,f_mw_mv_{\vec{w}}^{\frac{1-m}{m}}\right)\right|v_{\vec{w}}^{-1}>t\right\}\right)\lesssim [v_{\vec{w}}]_{A_1}^{2m+\frac{2m-2}{m}}t^{-\frac{1}{m}}\prod_{i=1}^m\|f_i\|_{L^1(w_i)}^{\frac{1}{m}}.$$ Without loss of generality, assume that $f_1,\ldots,f_m$ are continuous functions with compact support and that $\|f_1\|_{L^1(w_1)}=\cdots=\|f_m\|_{L^1(w_m)}=1$. Apply the Calder\'on-Zygmund decomposition to $\displaystyle f_iw_iv_{\vec{w}}^{-1}$ at height $t^{\frac{1}{m}}$ with respect to $v_{\vec{w}}dx$ to write $$f_iw_iv_{\vec{w}}^{-1}=g_i+b_i=g_i+\sum_{j=1}^{\infty}b_{i,j}$$ where the following properties hold:
\begin{enumerate}
\item $\|g_i\|_{L^{\infty}(\mathbb{R}^n)}\lesssim [v_{\vec{w}}]_{A_1}t^{\frac{1}{m}}$ and $\|g_i\|_{L^1(v_{\vec{w}})}\leq \|f_i\|_{L^1(w_i)}$,
\item the $b_{i,j}$ are supported on pairwise disjoint cubes $Q_{i,j}$ satisfying $$\sum_{j=1}^{\infty}v_{\vec{w}}(Q_{i,j})\leq t^{-\frac{1}{m}}\|f_i\|_{L^1(w_i)},$$
\item $\int_{Q_{i,j}}b_{i,j}(x)v_{\vec{w}}(x)dx=0$, 
\item $\|b_{i,j}\|_{L^1(v_{\vec{w}})}\lesssim [v_{\vec{w}}]_{A_1}t^{\frac{1}{m}}v_{\vec{w}}(Q_{i,j})$, and
\item $\|b_i\|_{L^1(v_{\vec{w}})}\lesssim \|f_i\|_{L^1(w_i)}$.
\end{enumerate}
Set
\begin{align*}
&S_1:=\left\{\left|T\left(g_1v_{\vec{w}}^{\frac{1}{m}},g_2v_{\vec{w}}^{\frac{1}{m}},\ldots,g_mv_{\vec{w}}^{\frac{1}{m}}\right)\right|v_{\vec{w}}^{-1}>\frac{t}{2^m}\right\},\\
&S_2:=\left\{\left|T\left(b_1v_{\vec{w}}^{\frac{1}{m}},g_2v_{\vec{w}}^{\frac{1}{m}},\ldots,g_mv_{\vec{w}}^{\frac{1}{m}}\right)\right|v_{\vec{w}}^{-1}>\frac{t}{2^m}\right\},\\
&S_3:=\left\{\left|T\left(g_1v_{\vec{w}}^{\frac{1}{m}},b_2v_{\vec{w}}^{\frac{1}{m}},\ldots,g_mv_{\vec{w}}^{\frac{1}{m}}\right)\right|v_{\vec{w}}^{-1}>\frac{t}{2^m}\right\},\\
&\quad\quad\quad\quad\quad\quad\quad\quad\cdots\\
&S_{2^m}:=\left\{\left|T\left(b_1v_{\vec{w}}^{\frac{1}{m}},b_2v_{\vec{w}}^{\frac{1}{m}},\ldots,b_mv_{\vec{w}}^{\frac{1}{m}}\right)\right|v_{\vec{w}}^{-1}>\frac{t}{2^m}\right\};
\end{align*}
where each $S_s=\left\{\left|T\left(h_1v_{\vec{w}}^{\frac{1}{m}},\ldots,h_mv_{\vec{w}}^{\frac{1}{m}}\right)\right|v_{\vec{w}}^{-1}>\frac{t}{2^m}\right\}$ with $h_i\in\left\{b_i,g_i\right\}$ and all the sets $S_s$ are distinct. Since $$v_{\vec{w}}\left(\left\{\left|T\left(f_1w_1v_{\vec{w}}^{\frac{1-m}{m}},\ldots,f_mw_mv_{\vec{w}}^{\frac{1-m}{m}}\right)\right|v_{\vec{w}}^{-1}>t\right\}\right)\leq \sum_{s=1}^{2^m}v_{\vec{w}}(S_s),$$ it suffices to control each $v_{\vec{w}}(S_s)$.

Use Chebyshev's inequality, the boundedness of $T$ from $(L^m(\mathbb{R}^n))^m$ to $L^{1}(\mathbb{R}^n)$ (which holds by Theorem 5), and property (1) to see
\begin{align*}
v_{\vec{w}}\left(S_1\right)&\lesssim t^{-1}\int_{\mathbb{R}^n}\left|T\left(g_1v_{\vec{w}}^{\frac{1}{m}},\ldots,g_mv_{\vec{w}}^{\frac{1}{m}}\right)(x)\right|dx\\
&\lesssim t^{-1}\prod_{i=1}^m\left(\int_{\mathbb{R}^n}|g_i(x)|^mv_{\vec{w}}(x)dx\right)^{\frac{1}{m}}\\
&\leq t^{-\frac{1}{m}}\prod_{i=1}^m\|g_i\|_{L^1(v_{\vec{w}})}^{\frac{1}{m}}\\
&\leq t^{-\frac{1}{m}}.
\end{align*}

Consider the set $S_s$ for a fixed $2\leq s\leq 2^m$. Suppose that there are $l$ functions of the form $b_i$ and $m-l$ functions of the form $g_i$ appearing as entries in the $T\left(h_1v_{\vec{w}}^{\frac{1}{m}},\ldots,h_mv_{\vec{w}}^{\frac{1}{m}}\right)$ involved in the definition of $S_s$. By symmetry, we may assume that the $b_i$ are in the first $l$ entries and the $g_i$ are in the remaining $m-l$ entries. Let $Q_{i,j}^*:=2\sqrt{n}Q_{i,j}$, $\Omega_i^*:=\bigcup_{j=1}^{\infty}Q_{i,j}^*$, and $\Omega^*:=\bigcup_{i=1}^m\Omega_i^*$.

By the doubling property of $v_{\vec{w}}dx$, the fact that the $Q_{i,j}$ are pairwise disjoint, and property (3), we have
$$v_{\vec{w}}\left(\Omega^*\right)\leq \sum_{i=1}^m\sum_{j=1}^{\infty}v_{\vec{w}}(Q_{i,j}^*)
\lesssim [v_{\vec{w}}]_{A_1}\sum_{i=1}^m\sum_{j=1}^{\infty}v_{\vec{w}}(Q_{i,j})
\leq [v_{\vec{w}}]_{A_1}t^{-\frac{1}{m}}.
$$
Therefore 
\begin{align*}
v_{\vec{w}}(S_s)&\leq v_{\vec{w}}(\Omega^*)+v_{\vec{w}}\left(\left\{\mathbb{R}^n\setminus \Omega^* : \left|T\left(b_1v_{\vec{w}}^{\frac{1}{m}},\ldots,b_lv_{\vec{w}}^{\frac{1}{m}},g_{l+1}v_{\vec{w}}^{\frac{1}{m}},\ldots,g_mv_{\vec{w}}^{\frac{1}{m}}\right)\right|v_{\vec{w}}^{-1}>\frac{t}{2^m}\right\}\right)\\
&\lesssim [v_{\vec{w}}]_{A_1}t^{-\frac{1}{m}}+v_{\vec{w}}\left(\left\{\mathbb{R}^n\setminus \Omega^* : \left|T\left(b_1v_{\vec{w}}^{\frac{1}{m}},\ldots,b_lv_{\vec{w}}^{\frac{1}{m}},g_{l+1}v_{\vec{w}}^{\frac{1}{m}},\ldots,g_mv_{\vec{w}}^{\frac{1}{m}}\right)\right|v_{\vec{w}}^{-1}>\frac{t}{2^m}\right\}\right).
\end{align*}
Now use Chebyshev's inequality, the fact that $\int_{Q_{i,j_i}}b_{i,j_i}(x)v_{\vec{w}}(x)dx=0$, and trivial bounds to estimate
\begin{align*}
v_{\vec{w}}&\left(\left\{\mathbb{R}^n\setminus \Omega^* : \left|T\left(b_1v_{\vec{w}}^{\frac{1}{m}},\ldots,b_lv_{\vec{w}}^{\frac{1}{m}},g_{l+1}v_{\vec{w}}^{\frac{1}{m}},\ldots,g_mv_{\vec{w}}^{\frac{1}{m}}\right)\right|v_{\vec{w}}^{-1}>\frac{t}{2^m}\right\}\right)\\
&\lesssim t^{-1}\int_{\mathbb{R}^n\setminus\Omega^*}\left|T\left(b_1v_{\vec{w}}^{\frac{1}{m}},\ldots,b_lv_{\vec{w}}^{\frac{1}{m}},g_{l+1}v_{\vec{w}}^{\frac{1}{m}},\ldots,g_mv_{\vec{w}}^{\frac{1}{m}}\right)(x)\right|dx
\end{align*}
\begin{align*}
&\leq t^{-1}\sum_{j_1,\ldots,j_l=1}^{\infty}\int_{\mathbb{R}^n\setminus\Omega^*}\left|\int_{\mathbb{R}^{n(m-l)}}\int_{Q_{1,j_1}}\cdots\int_{Q_{l,j_l}}K(x,\vv{y}_{1,m})\right.\\
&\quad\quad \left.\times\left(\prod_{i=1}^lb_{i,j_i}(y_i)\right)\left(\prod_{i=l+1}^mg_i(y_i)\right)\left(\prod_{i=1}^mv_{\vec{w}}(y_i)^{\frac{1}{m}}\right)d\vv{y}_{1,m}\right|dx\\
&\leq t^{-1}\sum_{j_1,\ldots,j_l=1}^{\infty}\left(\prod_{i=1}^l\sup_{Q_{i,j_i}}v_{\vec{w}}^{\frac{1-m}{m}}\right)\int_{\mathbb{R}^{n(m-l)}}\int_{Q_{l,j_l}}\cdots\int_{Q_{1,j_1}}\left(\prod_{i=1}^l|b_{i,j_i}(y_i)|v_{\vec{w}}(y_i)\right)\\
&\quad\quad \times\left(\prod_{i=l+1}^m|g_i(y_i)|v_{\vec{w}}(y_i)^{\frac{1}{m}}\right)\int_{\mathbb{R}^n\setminus\Omega^*}|K(x,\vv{y}_{1,m})-K(x,\vv{c}_{(1,j_1),(l,j_l)},\vv{y}_{l+1,m})|dxd\vv{y}_{1,m}\\
&\leq t^{-1}\sum_{j_1,\ldots,j_l=1}^{\infty}\left(\prod_{i=1}^l\sup_{Q_{i,j_i}}v_{\vec{w}}^{\frac{1-m}{m}}\right)\left(\prod_{i=1}^l\|b_{i,j_i}\|_{L^1(v_{\vec{w}})}\right)\left(\prod_{i=l+1}^m\|g_i\|_{L^{\infty}(\mathbb{R}^n)}\right)\\
&\quad\quad\times\int_{\mathbb{R}^{n(m-l)}}\sup_{\substack{(y_1,\ldots,y_l) \\ \in\prod_{i=1}^lQ_{i,j_i}}}\int_{\mathbb{R}^n\setminus\Omega^*}|K(x,\vv{y}_{1,m})-K(x,\vv{c}_{(1,j_1),(l,j_l)},\vv{y}_{l+1,m})|\\
&\quad\quad\times\left(\prod_{i=l+1}^mv_{\vec{w}}(y_i)^{\frac{1}{m}}\right)dxd\vv{y}_{l+1,m}.\\
\end{align*}
Apply property (4), property (1), the fact that $\left(\sup_{Q_{i,j_i}}v_{\vec{w}}^{\frac{1-m}{m}}\right)= \left(\inf_{Q_{i,j_i}}v_{\vec{w}}\right)^{\frac{1-m}{m}}$ , the $A_1$ condition of $v_{\vec{w}}$, and trivial estimates to bound the above expression by a constant times
\begin{align*}
&[v_{\vec{w}}]_{A_1}^m\sum_{j_1,\ldots,j_l=1}^{\infty}\left(\prod_{i=1}^l\sup_{Q_{i,j_i}}v_{\vec{w}}^{\frac{1-m}{m}}\right)\left(\prod_{i=1}^lv_{\vec{w}}(Q_{i,j_i})\right)\\
&\quad\times\int_{\mathbb{R}^{n(m-l)}}\sup_{\substack{(y_1,\ldots,y_l) \\ \in\prod_{i=1}^lQ_{i,j_i}}}\int_{\mathbb{R}^n\setminus\Omega^*}|K(x,\vv{y}_{1,m})-K(x,\vv{c}_{(1,j_1),(l,j_l)},\vv{y}_{l+1,m})|\\
&\quad\times\left(\prod_{i=l+1}^mv_{\vec{w}}(y_i)^{\frac{1}{m}}\right)d\vv{y}_{l+1,m}dx\\
&\leq [v_{\vec{w}}]_{A_1}^{m+l}\sum_{j_1,\ldots,j_l=1}^{\infty}\left(\prod_{i=1}^l\inf_{Q_{i,j_i}}v_{\vec{w}}\right)^{\frac{1-m}{m}}\left(\prod_{i=1}^l|Q_{i,j_i}|\left(\inf_{Q_{i,j_i}}v_{\vec{w}}\right)\right)\\
&\quad\times\int_{\mathbb{R}^{n(m-l)}}\sup_{\substack{(y_1,\ldots,y_l) \\ \in\prod_{i=1}^lQ_{i,j_i}}}\int_{\mathbb{R}^n\setminus\Omega^*}|K(x,\vv{y}_{1,m})-K(x,\vv{c}_{(1,j_1),(l,j_l)},\vv{y}_{l+1,m})|\\
&\quad\times\left(\prod_{i=l+1}^mw_i(y_i)^{\frac{1}{m}}\right)d\vv{y}_{l+1,m}dx
\end{align*}
\begin{align*}
&\leq [v_{\vec{w}}]_{A_1}^{2m}\sum_{j_1,\ldots,j_l=1}^{\infty}\left(\prod_{i=1}^l|Q_{i,j_i}|\left(\inf_{Q_{i,j_i}}v_{\vec{w}}^{\frac{1}{m}}\right)\right)\\
&\quad\times\int_{\mathbb{R}^{n(m-l)}}\sup_{\substack{(y_1,\ldots,y_l) \\ \in\prod_{i=1}^lQ_{i,j_i}}}\int_{\mathbb{R}^n\setminus\Omega^*}|K(x,\vv{y}_{1,m})-K(x,\vv{c}_{(1,j_1),(l,j_l)},\vv{y}_{l+1,m})|\\
&\quad\times\left(\prod_{i=l+1}^mw_i(y_i)^{\frac{1}{m}}\right)d\vv{y}_{l+1,m}dx\\
&\leq [v_{\vec{w}}]_{A_1}^{2m}\sum_{j_1,\ldots,j_l=1}^{\infty}\prod_{i=1}^lv_{\vec{w}}^{\frac{1}{m}}(Q_{i,j_i})\\
&\quad\times\int_{\mathbb{R}^{n(m-l)}}\sup_{\substack{(y_1,\ldots,y_l) \\ \in\prod_{i=1}^lQ_{i,j_i}}}\int_{\mathbb{R}^n\setminus\Omega^*}|K(x,\vv{y}_{1,m})-K(x,\vv{c}_{(1,j_1),(l,j_l)},\vv{y}_{l+1,m})|\\
&\quad\times\left(\prod_{i=l+1}^mw_i(y_i)^{\frac{1}{m}}\right)d\vv{y}_{l+1,m}dx.\\
\end{align*}
By Lemma 4 and property (2), the above expression is controlled by a constant times $$[v_{\vec{w}}]_{A_1}^{2m+\frac{2m-2}{m}}\sum_{i=1}^{l} v_{\vec{w}}\left(\bigcup_{j=1}^{\infty}Q_{i,j}\right)\lesssim [v_{\vec{w}}]_{A_1}^{2m+\frac{2m-2}{m}}t^{-\frac{1}{m}}.$$ Therefore $$v_{\vec{w}}(S_s)\leq [v_{\vec{w}}]_{A_1}t^{-\frac{1}{m}}+[v_{\vec{w}}]_{A_1}^{2m+\frac{2m-2}{m}}t^{-\frac{1}{m}}\lesssim [v_{\vec{w}}]_{A_1}^{2m+\frac{2m-2}{m}}t^{-\frac{1}{m}}.$$

Putting the previous estimates together gives $$v_{\vec{w}}\left(\left\{\left|T\left(f_1w_1^{\frac{m+1}{m}}v_{\vec{w}}^{-1},\ldots,f_mw_m^{\frac{m+1}{m}}v_{\vec{w}}^{-1}\right)\right|v_{\vec{w}}^{-1}>t\right\}\right)\lesssim t^{-\frac{1}{m}}+\sum_{s=2}^{2^m}[v_{\vec{w}}]_{A_1}^{2m+\frac{2m-2}{m}}t^{-\frac{1}{m}}$$ $$\lesssim[v_{\vec{w}}]_{A_1}^{2m+\frac{2m-2}{m}}t^{-\frac{1}{m}}.$$
\end{proof}

\begin{proof}[Proof 2]
Let $t>0$ be given. We will show that $$v_{\vec{w}}\left(\left\{\left|T\left(f_1w_1v_{\vec{w}}^{\frac{1-m}{m}},\ldots,f_mw_mv_{\vec{w}}^{\frac{1-m}{m}}\right)\right|v_{\vec{w}}^{-1}>t\right\}\right)\lesssim [v_{\vec{w}}]_{A_1}^{2m+\frac{2m-2}{m}}t^{-\frac{1}{m}}\prod_{i=1}^m\|f_i\|_{L^1(w_i)}^{\frac{1}{m}}.$$ Assume that $f_1,\ldots,f_m$ are nonnegative, continuous functions with compact support and that $\|f_1\|_{L^1(w_1)}=\cdots=\|f_m\|_{L^1(w_m)}=1$. Assume that $v_{\vec{w}}(\mathbb{R}^n) > t^{-\frac{1}{m}}$ (otherwise there is nothing to prove). Set \[\Omega_i:=\left\{M_{v_{\vec{w}}}\left(f_iw_iv_{\vec{w}}^{-1}\right)>t^{\frac{1}{m}}\right\}\,\,\,\,\,\, \text{and}\,\,\,\,\,\, \Omega:=\bigcup_{i=1}^m\Omega_i.\] Apply a Whitney decomposition to write $$\Omega_i=\bigcup_{j=1}^{\infty} Q_{i,j},$$ a disjoint union of dyadic cubes where $$2\text{diam}(Q_{i,j})\leq d(Q_{i,j},\mathbb{R}^n\setminus \Omega_i)\leq 8\text{diam}(Q_{i,j}).$$ Put \[g_i:=f_iw_iv_{\vec{w}}^{-1}\mathbbm{1}_{\mathbb{R}^n\setminus \Omega_i}, \quad\quad b_i:=f_iw_iv_{\vec{w}}^{-1}\mathbbm{1}_{\Omega_i}, \quad\quad \text{and}\quad\quad b_{i,j}:=f_iw_iv_{\vec{w}}^{-1}\mathbbm{1}_{Q_{i,j}}.\] Then $$f_iw_iv_{\vec{w}}^{-1}=g_i+b_i=g_i+\sum_{j=1}^{\infty}b_{i,j},$$ where 
\begin{enumerate}
    \item $\|g_i\|_{L^{\infty}(\mathbb{R}^n)}\leq t^{\frac{1}{m}}$ and $\|g_i\|_{L^1(v_{\vec{w}})}\leq \|f_i\|_{L^1(w_i)}$,
    \item the $b_{i,j}$ are supported on pairwise disjoint cubes $Q_{i,j}$ satisfying $$\sum_{j=1}^{\infty}v_{\vec{w}}(Q_{i,j})\lesssim t^{-\frac{1}{m}}\|f_i\|_{L^1(w_i)},$$
    \item $\|b_{i,j}\|_{L^1(v_{\vec{w}})}\leq (17\sqrt{n})^n[v_{\vec{w}}]_{A_1}t^{\frac{1}{m}}v_{\vec{w}}(Q_{i,j})$, and 
    \item $\|b_i\|_{L^1(v_{\vec{w}})}\leq \|f_i\|_{L^1(w_i)}$.
\end{enumerate}

To justify the above properties, since $$f_i(x)w_i(x)v_{\vec{w}}^{-1}(x)\leq M_{v_{\vec{w}}}\left(f_iw_iv_{\vec{w}}^{-1}\right)(x)\leq t^{\frac{1}{m}}$$ for almost every $x \in \mathbb{R}^n\setminus \Omega_i$, it is true that $\|g_i\|_{L^{\infty}(\mathbb{R}^n)}\leq t^{\frac{1}{m}}$. Noticing that $g_i$ is a restriction of $f_iw_iv_{\vec{w}}^{-1}$, we have $\|g_i\|_{L^1(v_{\vec{w}})}\leq \|f_i\|_{L^1(w_i)}$, so property (1) holds. We obtain property (2) using Lemma 2 as follows $$\sum_{j=1}^{\infty}v_{\vec{w}}(Q_{i,j})=v_{\vec{w}}(\Omega_i)\lesssim t^{-\frac{1}{m}}.$$

Addressing (3), for a fixed $Q_{i,j}$, let $Q_{i,j}^*=17\sqrt{n}Q_{i,j}$. Then $Q_{i,j}^*\cap (\mathbb{R}^n\setminus G_i) \neq \emptyset$, so there is a point $x\in Q_{i,j}^*$ such that $M_{v_{\vec{w}}}(f_iw_iv_{\vec{w}}^{-1})(x)\leq t^{\frac{1}{m}}$. In particular, $\int_{Q_{i,j}^*}f_i(y)w_i(y)dy\leq t^{\frac{1}{m}}v_{\vec{w}}(Q_{i,j}^*)$. Since $v_{\vec{w}}(Q_{i,j}^*)\leq (17\sqrt{n})^n[v_{\vec{w}}]_{A_1}v_{\vec{w}}(Q_{i,j})$, we have $$\|b_{i,j}\|_{L^1(v_{\vec{w}})}=\int_{Q_{i,j}}f_i(y)w_i(y)dy\leq\int_{Q_{i,j}^*}f_i(y)w_i(y)dy$$ $$\leq t^{\frac{1}{m}}v_{\vec{w}}(Q_{i,j}^*) \leq (17\sqrt{n})^n[v_{\vec{w}}]_{A_1}t^{\frac{1}{m}}v_{\vec{w}}(Q_{i,j}),$$ proving (3). Property (4) follows since $b_i$ is a restriction of $f_iw_iv_{\vec{w}}^{-1}$.

Set
\begin{align*}
S_1&:=\left\{\left|T\left(g_1v_{\vec{w}}^{\frac{1}{m}},g_2v_{\vec{w}}^{\frac{1}{m}},\ldots,g_mv_{\vec{w}}^{\frac{1}{m}}\right)\right|v_{\vec{w}}^{-1}>\frac{t}{2^m}\right\},\\
S_2&:=\left\{\left|T\left(b_1v_{\vec{w}}^{\frac{1}{m}},g_2v_{\vec{w}}^{\frac{1}{m}},\ldots,g_mv_{\vec{w}}^{\frac{1}{m}}\right)\right|v_{\vec{w}}^{-1}>\frac{t}{2^m}\right\},\\
S_3&:=\left\{\left|T\left(g_1v_{\vec{w}}^{\frac{1}{m}},b_2v_{\vec{w}}^{\frac{1}{m}},\ldots,g_mv_{\vec{w}}^{\frac{1}{m}}\right)\right|v_{\vec{w}}^{-1}>\frac{t}{2^m}\right\},\\
&\quad\quad\quad\quad\quad\quad\quad\quad\cdots\\
S_{2^m}&:=\left\{\left|T\left(b_1v_{\vec{w}}^{\frac{1}{m}},b_2v_{\vec{w}}^{\frac{1}{m}},\ldots,b_mv_{\vec{w}}^{\frac{1}{m}}\right)\right|v_{\vec{w}}^{-1}>\frac{t}{2^m}\right\};
\end{align*}
where each $S_s=\left\{\left|T\left(h_1v_{\vec{w}}^{\frac{1}{m}},\ldots,h_mv_{\vec{w}}^{\frac{1}{m}}\right)\right|v_{\vec{w}}^{-1}>\frac{t}{2^m}\right\}$ with $h_i\in\left\{b_i,g_i\right\}$ and all the sets $S_s$ are distinct. Since \[v_{\vec{w}}\left(\left\{\left|T\left(f_1w_1v_{\vec{w}}^{\frac{1-m}{m}},\ldots,f_mw_mv_{\vec{w}}^{\frac{1-m}{m}}\right)\right|v_{\vec{w}}^{-1}>t\right\}\right)\leq \sum_{s=1}^{2^m}v_{\vec{w}}(S_s),\] it suffices to control each $v_{\vec{w}}(S_s)$.

Use Chebyshev's inequality, the boundedness of $T$ from $(L^m(\mathbb{R}^n))^m$ to $L^{1}(\mathbb{R}^n)$ (which holds by Theorem 5), and property (1) to see
\begin{align*}
v_{\vec{w}}(S_1)&\lesssim t^{-1}\int_{\mathbb{R}^n}\left|T\left(g_1v_{\vec{w}}^{\frac{1}{m}},\ldots,g_mv_{\vec{w}}^{\frac{1}{m}}\right)(x)\right|dx\\
&\lesssim t^{-1}\prod_{i=1}^m\left(\int_{\mathbb{R}^n}|g_i(x)|^mv_{\vec{w}}(x)dx\right)^{\frac{1}{m}}\\
&\leq t^{-\frac{1}{m}}\prod_{i=1}^m\|g_i\|_{L^1(v_{\vec{w}})}^{\frac{1}{m}}\\
&\leq t^{-\frac{1}{m}}.
\end{align*}

Consider the set $S_s$ for a fixed $2\leq s\leq 2^m$. Suppose that there are $l$ functions of the form $b_i$ and $m-l$ functions of the form $g_i$ appearing as entries in the $T\left(h_1w_1^{\frac{1}{m}},\ldots,h_mw_m^{\frac{1}{m}}\right)$ involved in the definition of $S_s$. By symmetry, we may assume that the $b_i$ are in the first $l$ entries and the $g_i$ are in the remaining $m-l$ entries.

Let $c_{i,j}$ denote the center of $Q_{i,j}$ and let $a_{i,j}=\|b_{i,j}\|_{L^1(v_{\vec{w}})}(17\sqrt{n})^{-n}[v_{\vec{w}}]_{A_1}^{-1}$. Set $$E_{i,j}:=Q(c_{i,j},r_{i,j}),$$ where $r_{i,j}>0$ is chosen so that $v_{\vec{w}}(E_{i,j})=a_{i,j}t^{-\frac{1}{m}}$. Note that such $E_{i,j}$ exist since the function $r \mapsto v_{\vec{w}}(Q(x,r))$ increases to $v_{\vec{w}}(\mathbb{R}^n)>t^{-\frac{1}{m}}$ as $r \rightarrow \infty$, approaches $0$ as $r\rightarrow 0$, and is continuous from the right for almost every $x \in \mathbb{R}^n$. Using property (3), we see $$v_{\vec{w}}(E_{i,j})=a_{i,j}t^{-\frac{1}{m}}\leq v_{\vec{w}}(Q_{i,j}).$$ Since $E_{i,j}$ is a cube with the same center as $Q_{i,j}$ and since $v_{\vec{w}}(E_{i,j})\leq v_{\vec{w}}(Q_{i,j}),$ it is true that $E_{i,j}\subseteq Q_{i,j}$. Define $$E_i:=\bigcup_{j=1}^{\infty}E_{i,j}.$$

For $k=0,\ldots, l$, define $$\sigma_{k}:=T\left((17\sqrt{n})^n[v_{\vec{w}}]_{A_1}t^{\frac{1}{m}}\vv{\mathbbm{1}_{E_{1,k}}}v_{\vec{w}}^{\frac{1}{m}},\vv{b_{k+1,l}}v_{\vec{w}}^{\frac{1}{m}},\vv{g_{l+1,m}}v_{\vec{w}}^{\frac{1}{m}}\right).$$ Then, by adding and subtracting $\sigma_k$ for $1\leq k \leq l$, we have 
\begin{align*}
v_{\vec{w}}\left(S_s\right)&\leq\sum_{k=1}^{l}v_{\vec{w}}\left(\left\{\left|\sigma_{k-1}-\sigma_k\right|>\frac{t}{(l+1)2^m}\right\}\right)+v_{\vec{w}}\left(\left\{|\sigma_l|>\frac{t}{(l+1)2^m}\right\}\right)\\
&\lesssim v_{\vec{w}}(\Omega)+\sum_{k=1}^{l}v_{\vec{w}}\left(\left\{\mathbb{R}^n\setminus \Omega: \left|\sigma_{k-1}-\sigma_k\right|>\frac{t}{(l+1)2^m}\right\}\right)\\
&\quad\quad+v_{\vec{w}}\left(\left\{|\sigma_l|>\frac{t}{(l+1)2^m}\right\}\right).
\end{align*}
Using property (2), we have $$v_{\vec{w}}(\Omega)\leq\sum_{i=1}^m\sum_{j=1}^{\infty}v_{\vec{w}}(Q_{i,j})\lesssim t^{-\frac{1}{m}},$$ therefore
\begin{align*}
v_{\vec{w}}\left(S_s\right)&\lesssim t^{-\frac{1}{m}}+\sum_{k=1}^{l}v_{\vec{w}}(P_k)+v_{\vec{w}}(P),
\end{align*}
where
\begin{align*}
&P_k:=\left\{\mathbb{R}^n\setminus \Omega: \left|\sigma_{k-1}-\sigma_k\right|>\frac{t}{(l+1)2^m}\right\},\quad\text{and}\\
&P:=\left\{|\sigma_l|>\frac{t}{(l+1)2^m}\right\}.
\end{align*}

We will first estimate $v_{\vec{w}}(P_k)$ for $k\in\{1,\ldots,l\}$. Notice that 
$$\sigma_{k-1}(x)-\sigma_k(x)$$ $$=T\left((17\sqrt{n})^n[v_{\vec{w}}]_{A_1}t^{\frac{1}{m}}\vv{\mathbbm{1}_{E_{i,k-1}}}v_{\vec{w}}^{\frac{1}{m}},\left(b_k-(17\sqrt{n})^n[v_{\vec{w}}]_{A_1}t^{\frac{1}{m}}\mathbbm{1}_{E_k}\right)v_{\vec{w}}^{\frac{1}{m}},\vv{b_{k+1,l}}v_{\vec{w}}^{\frac{1}{m}},\vv{g_{l+1,m}}v_{\vec{w}}^{\frac{1}{m}}\right)(x).$$
Begin by using Chebyshev's inequality, the fact that $$\int_{\mathbb{R}^n}\left(b_{k,j_k}(x)-(17\sqrt{n})^n[v_{\vec{w}}]_{A_1}t^{\frac{1}{m}}\mathbbm{1}_{E_{k,j_k}}(x)\right)v_{\vec{w}}(x)dx=0,$$ and trivial estimates to see
\begin{align*}
v_{\vec{w}}(P_k)&\lesssim t^{-1}\int_{\mathbb{R}^n\setminus \Omega}\left|\sigma_{k-1}(x)-\sigma_k(x)\right|dx\\
&\leq t^{-1}\sum_{j_1,\ldots,j_l=1}^{\infty}\int_{\mathbb{R}^n\setminus \Omega}\left|\int_{\mathbb{R}^{n(m-l)}}\int_{Q_{l,j_l}}\cdots\int_{Q_{1,j_1}}K(x,\vv{y}_{1,m})\right)\\
&\quad\quad \times \left(\prod_{i=1}^{k-1}(17\sqrt{n})^n[v_{\vec{w}}]_{A_1}t^{\frac{1}{m}}\mathbbm{1}_{E_{i,i,j_i}}(y_i)v_{\vec{w}}(y_i)^{\frac{1}{m}}\right)\\
&\quad\quad\times\left(\left(b_{k,j_k}(y_k)-(17\sqrt{n})^n[v_{\vec{w}}]_{A_1}t^{\frac{1}{m}}\mathbbm{1}_{E_{k,j_k}}(y_k)\right)v_{\vec{w}}(y_k)^{\frac{1}{m}}\right)\\
&\quad\quad\times \left.\left(\prod_{i=k+1}^lb_{i,i,j_i}(y_i)v_{\vec{w}}(y_i)^{\frac{1}{m}}\right)\left(\prod_{i=l+1}^mg_i(y_i)v_{\vec{w}}(y_i)^{\frac{1}{m}}\right)d\vv{y}_{1,m}\right|dx\\
&\leq t^{\frac{k-m-1}{m}}\sum_{j_1,\ldots,j_l=1}^{\infty}\left(\prod_{i=1}^l \sup_{Q_{i,j_i}}v_{\vec{w}}^{\frac{1-m}{m}}\right)\int_{\mathbb{R}^n\setminus \Omega}\int_{\mathbb{R}^{n(m-l)}}\int_{E_{l,j_l}}\cdots\int_{E_{1,j_1}}\\
&\quad\quad \times\left|K(x,\vv{y}_{1,m})-K(x,\vv{c}_{(1,j_1),(l,j_l)},\vv{y}_{l+1,m})\right|\left(\prod_{i=1}^{k-1}v_{\vec{w}}(y_i)\right)\\
&\quad\quad\times\left(\left|b_{k,j_k}(y_k)-(17\sqrt{n})^n[v_{\vec{w}}]_{A_1}t^{\frac{1}{m}}\mathbbm{1}_{E_{k,j_k}}(y_k)\right|v_{\vec{w}}(y_k)\right)\left(\prod_{i=k+1}^lb_{i,j_i}(y_i)v_{\vec{w}}(y_i)\right)\\
&\quad\quad\times\left(\prod_{i=l+1}^mg_i(y_i)v_{\vec{w}}(y_i)^{\frac{1}{m}}\right)d\vv{y}_{1,m}dx.
\end{align*}

Next use the fact that $v_{\vec{w}}(E_{i,j_i})\lesssim t^{-\frac{1}{m}}\|b_{i,j_i}\|_{L^1(v_{\vec{w}})}$, Fubini's theorem, trivial estimates, the fact that $$\left\|b_{k,j_k}-(17\sqrt{n})^n[v_{\vec{w}}]_{A_1}t^{\frac{1}{m}}\mathbbm{1}_{E_{k,j_k}}\right\|_{L^1(v_{\vec{w}})}\lesssim \left\|b_{k,j_k}\right\|_{L^1(v_{\vec{w}})},$$ and property (1) to control 
\begin{align*}
v_{\vec{w}}(P_k)&\lesssim t^{-1}\sum_{j_1,\ldots,j_l=1}^{\infty}\left(\prod_{i=1}^l \sup_{Q_{i,j_i}}v_{\vec{w}}^{\frac{1-m}{m}}\right)\left(\prod_{i=1}^{k-1}\|b_{i,j_i}\|_{L^1(v_{\vec{w}})}\right)\left\|b_{k,j_k}-(17\sqrt{n})^n[v_{\vec{w}}]_{A_1}t^{\frac{1}{m}}\mathbbm{1}_{E_{k,j_k}}\right\|_{L^1(v_{\vec{w}})}\\
&\quad\times\left(\prod_{i=k+1}^l\|b_{i,j_i}\|_{L^1(v_{\vec{w}})}\right)\left(\prod_{i=l+1}^m\|g_i\|_{L^{\infty}(\mathbb{R}^n)}\right)\int_{\mathbb{R}^{n(m-l)}}\sup_{\substack{(y_1,\ldots,y_l) \\\in \prod_{i=1}^lQ_{i,j_i}}} \int_{\mathbb{R}^n\setminus \Omega}\\
&\quad\times|K(x,\vv{y}_{1,m})-K(x,\vv{c}_{(1,j_1),(l,j_l)},\vv{y}_{l+1,m})|\left(\prod_{i=l+1}^mv_{\vec{w}}(y_i)^{\frac{1}{m}}\right)dxd\vv{y}_{l+1,m}
\end{align*}
\begin{align*}
&\lesssim t^{-\frac{l}{m}}\sum_{j_1,\ldots,j_l=1}^{\infty}\left(\prod_{i=1}^l \sup_{Q_{i,j_i}}v_{\vec{w}}^{\frac{1-m}{m}}\right)\left(\prod_{i=1}^l\|b_{i,j_i}\|_{L^1(v_{\vec{w}})}\right)\\
&\quad\times\int_{\mathbb{R}^{n(m-l)}}\sup_{\substack{(y_1,\ldots,y_l) \\\in \prod_{i=1}^lQ_{i,j_i}}} \int_{\mathbb{R}^n\setminus \Omega} |K(x,\vv{y}_{1,m})-K(x,\vv{c}_{(1,j_1),(l,j_l)},\vv{y}_{l+1,m})|\\
&\quad\times\left(\prod_{i=l+1}^mv_{\vec{w}}(y_i)^{\frac{1}{m}}\right)dxd\vv{y}_{l+1,m}.
\end{align*}

Use property (3), the fact that $\left(\sup_{Q_{i,j_i}}v_{\vec{w}}^{\frac{1-m}{m}}\right)= \left(\inf_{Q_{i,j_i}}v_{\vec{w}}\right)^{\frac{1-m}{m}}$, the $A_1$ condition of $v_{\vec{w}}$, and trivial estimates to estimate $v_{\vec{w}}(P_k)$
\begin{align*}
v_{\vec{w}}(P_k)&\lesssim [v_{\vec{w}}]_{A_1}^{l}\sum_{j_1,\ldots,j_l=1}^{\infty}\left(\prod_{i=1}^l \sup_{Q_{i,j_i}}v_{\vec{w}}^{\frac{1-m}{m}}\right)\left(\prod_{i=1}^lv_{\vec{w}}(Q_{i,j_i})\right)\\
&\quad\times\int_{\mathbb{R}^{n(m-l)}}\sup_{\substack{(y_1,\ldots,y_l) \\\in \prod_{i=1}^lQ_{i,j_i}}} \int_{\mathbb{R}^n\setminus \Omega} |K(x,\vv{y}_{1,m})-K(x,\vv{c}_{(1,j_1),(l,j_l)},\vv{y}_{l+1,m})|\\
&\quad\times\left(\prod_{i=l+1}^mv_{\vec{w}}(y_i)^{\frac{1}{m}}\right)dxd\vv{y}_{l+1,m}\\
&\leq[v_{\vec{w}}]_{A_1}^{l}\sum_{j_1,\ldots,j_l=1}^{\infty}\left(\prod_{i=1}^l \inf_{ Q_{i,j_i}}v_{\vec{w}}\right)^{\frac{1-m}{m}}\left(\prod_{i=1}^lv_{\vec{w}}(Q_{i,j_i})\right)\\
&\quad\times\int_{\mathbb{R}^{n(m-l)}}\sup_{\substack{(y_1,\ldots,y_l) \\\in \prod_{i=1}^l Q_{i,j_i}}} \int_{\mathbb{R}^n\setminus \Omega} |K(x,\vv{y}_{1,m})-K(x,\vv{c}_{(1,j_1),(l,j_l)},\vv{y}_{l+1,m})|\\
&\quad\times\left(\prod_{i=l+1}^mv_{\vec{w}}(y_i)^{\frac{1}{m}}\right)dxd\vv{y}_{l+1,m}\\
&\leq [v_{\vec{w}}]_{A_1}^{2l}\sum_{j_1,\ldots,j_l=1}^{\infty}\left(\prod_{i=1}^l \left|Q_{i,j_i}\right|\left(\inf_{Q_{i,j_i}}v_{\vec{w}}^{\frac{1}{m}}\right)\right)\\
&\quad\times\int_{\mathbb{R}^{n(m-l)}}\sup_{\substack{(y_1,\ldots,y_l) \\\in \prod_{i=1}^l Q_{i,j_i}}} \int_{\mathbb{R}^n\setminus \Omega} |K(x,\vv{y}_{1,m})-K(x,\vv{c}_{(1,j_1),(l,j_l)},\vv{y}_{l+1,m})|\\
&\quad\times\left(\prod_{i=l+1}^mv_{\vec{w}}(y_i)^{\frac{1}{m}}\right)dxd\vv{y}_{l+1,m}
\end{align*}
\begin{align*}
&\leq [v_{\vec{w}}]_{A_1}^{2m}\sum_{j_1,\ldots,j_l=1}^{\infty}\left(\prod_{i=1}^l v_{\vec{w}}^{\frac{1}{m}}(Q_{i,j_i})\right)\\
&\quad\times\int_{\mathbb{R}^{n(m-l)}}\sup_{\substack{(y_1,\ldots,y_l) \\\in \prod_{i=1}^l Q_{i,j_i}}} \int_{\mathbb{R}^n\setminus G} |K(x,\vv{y}_{1,m})-K(x,\vv{c}_{(1,j_1),(l,j_l)},\vv{y}_{l+1,m})|\\
&\quad\times\left(\prod_{i=l+1}^mv_{\vec{w}}(y_i)^{\frac{1}{m}}\right)dxd\vv{y}_{l+1,m}.\\
\end{align*}
Use Lemma 4 and property (2) to finish the estimate $$v_{\vec{w}}(P_k)\lesssim [v_{\vec{w}}]_{A_1}^{2m+\frac{2m-2}{m}}\sum_{i=1}^{l}v_{\vec{w}}(\Omega)\lesssim[v_{\vec{w}}]_{A_1}^{2m+\frac{2m-2}{m}+1}t^{-\frac{1}{m}}.$$

The control of $v_{\vec{w}}(P)$ follows from Chebyshev's inequality, construction of the sets $E_i$, property (1), and property (4):
\begin{align*}
v_{\vec{w}}(P)&\lesssim t^{-1}\int_{\mathbb{R}^n}\left|\sigma_l(x)\right|dx\\
&\lesssim t^{-1+\frac{l}{m}}\left(\prod_{i=1}^lv_{\vec{w}}(E_i)^{\frac{1}{m}}\right)\left(\prod_{i=l+1}^m\left(\int_{\mathbb{R}^n}g_i(x)^mv_{\vec{w}}(x)dx\right)^{\frac{1}{m}}\right)\\
&\lesssim t^{-\frac{1}{m}}\left(\prod_{i=1}^l\|b_i\|_{L^1(v_{\vec{w}})}^{\frac{1}{m}}\right)\left(\prod_{i=l+1}^m\|g_i\|_{L^1(v_{\vec{w}})}^{\frac{1}{m}}\right)\\
&\leq t^{-\frac{1}{m}}.
\end{align*}

Put the estimates of $v_{\vec{w}}(P_k)$ and $v_{\vec{w}}(P)$ together to get
$$v_{\vec{w}}(S_s)\lesssim [v_{\vec{w}}]_{A_1}t^{-\frac{1}{m}}+\sum_{k=1}^{l}[v_{\vec{w}}]_{A_1}^{2m+\frac{2m-2}{m}+1}t^{-\frac{1}{m}}+t^{-\frac{1}{m}} \lesssim [v_{\vec{w}}]_{A_1}^{2m+\frac{2m-2}{m}+1}t^{-\frac{1}{m}}.$$

Finally, use the estimates of $v_{\vec{w}}(S_s)$, $1\leq s \leq 2^m$ to complete the proof 
$$v_{\vec{w}}\left(\left\{\left|T\left(f_1w_1v_{\vec{w}}^{\frac{1-m}{m}},\ldots,f_mw_mv_{\vec{w}}^{\frac{1-m}{m}}\right)\right|v_{\vec{w}}^{-1}>t\right\}\right)\leq |S_1|+\sum_{s=2}^{2^m}|S_s|$$ $$\lesssim t^{-\frac{1}{m}} +\sum_{s=2}^{2^m}[v_{\vec{w}}]_{A_1}^{2m+\frac{2m-2}{m}+1}t^{-\frac{1}{m}}\lesssim [v_{\vec{w}}]_{A_1}^{2m+\frac{2m-2}{m}+1}t^{-\frac{1}{m}}.$$
\end{proof}



\begin{bibdiv}
\begin{biblist}
\bib{CRR2018}{article}{
title={A sparse approach to mixed weak type inequalities},
author={M. Caldarelli},
author={I. P. Rivera-R\'ios},
date={December 2018},
journal={ArXiv e-prints},
}

\bib{CUMP2005}{article}{
title={Weighted weak-type inequalities and a conjecture of Sawyer},
author={D. Cruz-Uribe},
author={J. M. Martell},
author={C. P\'erez},
date={2005},
journal={Internat. Math. Res. Notices},
volume={30},
pages={1849--1871},
}

\bib{DLP2015}{article}{
title={Sharp weighted bounds for multilinear maximal functions and Calder\'on-Zygmund operators},
author={W. Dami\'an},
author={A. K. Lerner},
author={C. P\'erez},
date={2015},
journal={J. Fourier Anal. Appl.},
volume={21},
number={1},
pages={161--181},
}

\bib{Grafakos1}{book}{
title={Classical Fourier analysis},
author={L. Grafakos},
publisher={Springer},
edition={Third edition},
volume={249},
address={New York},
series={Graduate Texts in Mathematics},
date={2014}
}

\bib{Grafakos2}{book}{
title={Modern Fourier analysis},
author={L. Grafakos},
publisher={Springer},
edition={Third edition},
volume={250},
address={New York},
series={Graduate Texts in Mathematics},
date={2014}
}

\bib{GT2002}{article}{
title={Multilinear Calder\'on-Zygmund theory},
author={L. Grafakos},
author={R. Torres},
journal={Adv. Math.},
volume={165},
date={2002},
number={1},
pages={124--164}
}

\bib{HLYY2012}{article}{
title={Boundedness of Calder\'on-Zygmund operators on non-homogeneous metric measure spaces},
author={T. Hyt\"onen},
author={S. Liu},
author={D. Yang},
author={D. Yang},
date={2012},
journal={Canad. J. Math.},
volume={64},
number={4},
pages={892--923},
}

\bib{LOP2017}{article}{
title={Proof of an extension of E. Sawyer's conjecture about weighted mixed weak-type estimates},
author={K. Li},
author={S. Ombrosi},
author={C. P\'erez},
journal={Math. Ann.},
volume={374},
date={2019},
number={1-2},
pages={907--929}
}

\bib{LOPTTG2009}{article}{
title={New maximal functions and multiple weights for the multilinear Calder\'on-Zygmund theory},
author={A. K. Lerner},
author={S. Ombrosi},
author={C. P\'erez},
author={R. H. Torres},
author={R. Trujillo-Gonz\'alez},
journal={Adv. Math},
volume={220},
date={2009},
number={4},
pages={1222--1264}
}

\bib{LOP2019}{article}{
title={Weighted mixed weak-type inequalities for multilinear operators},
author={K. Li},
author={S. J. Ombrosi},
author={B. Picardi},
date={2019},
journal={Studia Math.},
volume={244},
number={2},
pages={203--215},
}

\bib{L2016}{article}{
title={Endpoint estimates for multilinear singular integral operators},
author={Y. Lin},
date={2016},
journal={Georgian Math. J.},
volume={23},
number={4}
pages={559--570},
}

\bib{MN2009}{article}{
title={Weighted norm inequalities for paraproducts and bilinear pseudodifferential operators with mild regularity},
author={D. Maldonado},
author={V. Naibo},
date={2009},
journal={J. Fourier Anal. Appl.},
volume={15},
number={2},
pages={218--261},
}

\bib{NTV1998}{article}{
title={Weak type estimates and Cotlar inequalities for Calder\'on-Zygmund operators on nonhomogeneous spaces},
author={F. Nazarov},
author={S. Treil},
author={A. Volberg},
journal={Internat. Math. Res. Notices},
volume={9},
date={1998},
number={9},
pages={463--487}
}

\bib{OP2016}{article}{
title={Mixed weak type estimates: examples and counterexamples related to a problem of E. Sawyer},
author={S. Ombrosi},
author={C. P\'erez},
volume={145},
date={2016},
journal={Colloq. Math.},
number={2},
pages={259--272},
}

\bib{OPR2016}{article}{
title={Quantitative weighted mixed weak-type inequalities for classical operators},
author={S. Ombrosi},
author={C. P\'erez},
author={J. Recchi},
journal={Indiana U. Math. J.},
volume={65},
date={2016},
number={2},
pages={615--640}
}

\bib{PT2014}{article}{
title={Minimal regularity conditions for the end-point estimate of bilinear Calder\'on-Zygmund operators},
author={C. P\'erez},
author={R. H. Torres},
date={2014},
journal={Proc. Amer. Math. Soc. Ser. B},
volume={1},
pages={1--13},
}

\bib{Stein}{book}{
title={Singular Integrals and Differentiability Properties of Functions},
author={E. M. Stein},
publisher={Princeton Univ. Press},
date={1970},
address={Princeton, NJ},
}

\bib{SW1958}{article}{
title={Interpolation of operators with change of measure},
author={E. M. Stein},
author={G. Weiss},
journal={Trans. Amer. Math. Soc.},
date={1958},
volume={87},
pages={159-172},
}

\bib{S2018}{article}{
title={A different approach to endpoint weak-type estimates for Calder\'on-Zygmund operators},
author={C. B. Stockdale},
journal={ArXiv e-prints},
date={December 2018},
}

\bib{SW2019}{article}{
title={An endpoint weak-type estimate for multilinear Calder\'on-Zygmund operators},
author={C. B. Stockdale},
author={B. D. Wick},
journal={J Fourier Anal Appl},
date={2019},
pages={https://doi.org/10.1007/s00041-019-09676-y},
}
\end{biblist}
\end{bibdiv}
\end{document}